\documentclass{amsart}
\usepackage{graphicx} 
\usepackage{amsmath,amsthm,mathrsfs,url,hyperref,stmaryrd}
\usepackage{biblatex}
\title{Zeta zeros of function fields heuristically equidistribute as genus goes to infinity}
\author{Julian Shah}
\newcommand{\Z}{\mathbb{Z}}
\newcommand{\F}{\mathbb{F}}
\newcommand{\Q}{\mathbb{Q}}
\newcommand{\C}{\mathbb{C}}
\newcommand{\M}{\mathcal{M}}
\newcommand{\J}{\mathcal{J}}
\newcommand{\Jac}{\mathcal{J}ac}
\newcommand{\Jacan}{\mathsf{Jac}}

\newcommand{\Spec}{\operatorname{Spec}}
\newcommand{\an}{\mathrm{an}}
\newcommand{\Frob}{\mathrm{Frob}}
\newcommand{\Tr}{\mathrm{Tr}}

\newtheorem{theorem}{Theorem}[section]
\newtheorem{corollary}[theorem]{Corollary}
\newtheorem{proposition}[theorem]{Proposition}
\newtheorem{lemma}[theorem]{Lemma}
\newtheorem{heuristic}[theorem]{Heuristic}
\newtheorem{theoremalpha}{Theorem}

\begin{document}

\maketitle

\vspace{-1cm}
\begin{abstract}
    We propose a heuristic for the growth rate of the average class number of curves in $\mathcal M_g(\mathbb F_q)$ as $g \to \infty$. We show that this heuristic implies that the zeros of the zeta functions of such curves become equidistributed on the circle as $g\to\infty$. Our heuristic is based on the assumption that the stable cohomology of the universal Jacobian dominates a point count using the trace formula, following the method of Achter et al.
\end{abstract}

\section{Introduction}

The study of the distribution of zeros of $L$-functions is of fundamental importance in number theory and arithmetic statistics. In this article we will focus on the case of the zeta functions attached to curves over finite fields, i.e. the zeta functions of function fields. For our purposes, all curves will be smooth and proper over a base; at various points we will have some cause to consider curves over a base that is not the spectrum of a field, though our main interest will be in curves of the above type over $\F_q$ which are geometrically connected. Such curves are in bijection with the $\F_q$-points of $\M_g$.

For such a curve $C$ of genus $g$, we can write its zeta function as
$$Z(C,t) = \frac{\prod_{i=1}^{2g}(1-\omega_it)}{(1-t)(1-qt)}$$
where the $\omega_i$ are complex numbers of absolute value $\sqrt q$\footnote{We will refer to the $\omega_i$ as the zeta zeros of $C$, despite the fact that it is the $\omega_i^{-1}$ that are the zeros of $Z(C,t)$; since we are mainly interested in the distribution of zeta zeros on the circle, we will allow ourselves this abuse of language.}. Considering that $Z(C,t)$ has only finitely many zeros, the problem of studying their distribution becomes most interesting when studying families of curves. Our goal is to show that a heuristic about the cohomology of moduli spaces yields a conjectural distribution for zeta zeros of curves ranging in all of $\M_g(\F_q)$ in the limit as $g\to\infty$. The method we apply, which is not new, is due to Achter, Erman, Kedlaya, Wood, and Zureick-Brown. We refer to their paper \cite{achter-wood}, which produces a conjectural distribution for the number of rational points on a random curve in $\M_g$, for an excellent exposition of the technique.

\subsection{Background} The foundational work in this direction was initiated by Katz and Sarnak in \cite{katz-sarnak}, in which they also consider the family $\M_g(\F_q)$ of all curves of genus $g$ over $\F_q$. They showed that in the limit as both $g$ and $q$ tend to infinity, the spacing of the $\omega_i$ on the circle tends toward the GUE measure, which is a measure from random matrix theory defined in terms of families of compact classical matrix groups. A key input in their work is Deligne's equidistribution theorem for the Frobenius operator acting on $H^1(C,\Q_\ell)$ whose eigenvalues are precisely the $\omega_i$. However, this approach has the drawback of requiring that the size $q$ of the base field go to infinity in order to obtain an equidistribution statement.

As such, if one wishes to understand how zeta zeros distribute for families of curves in which the genus may go to infinity but $q$ is to remain fixed, different methods are required. Much progress has been made for families of curves in which the gonality of the curves remains fixed as the genus grows, e.g. \cite{cyclic-p-fold}, \cite{hyperelliptic}, \cite{trigonal}; in these settings the method typically relies on being able to precisely describe and vary the geometry of covers of a given kind, and using such a description to make an explicit count.

If we want to understand the behavior of \textit{all} curves over $\F_q$, then such a method is no longer viable, as the stratum of curves of a given gonality in $\M_g$ will have positive codimension as $g\to\infty$. The increased difficulty in probing all of $\M_g$ can be attributed to $\M_g$ being far more geometrically complex than the moduli of curves of a fixed, low gonality; indeed, we currently have no analogous method to explicitly vary among all curves in $\M_g(\F_q)$. At present, such a sufficient understanding of $\M_g$ so as to make unconditional statements about various limiting numerics as $g\to\infty$ seems to be out of reach.

\subsection{Cohomological heuristics}\label{achter} A want for more geometric understanding of $\M_g$ was the impetus for the heuristic made in \cite{achter-wood} by Achter et al. We briefly review their idea here, and defer a more thorough motivation for the key heuristic to \textit{loc cit}.

The Behrend-Grothendieck-Lefschetz trace formula states that for a smooth, finite-type, equidimensional, Deligne-Mumford stack $\mathcal X$ over $\F_q$, the number of $\F_q$-rational points of $\mathcal X$ can be computed as
$$\#\mathcal X(\F_q) = \sum_{i=0}^{2\dim \mathcal X}(-1)^i\Tr(\Frob_q,H^i_c(\mathcal X_{\overline\F_q},\Q_\ell)),$$
where $H^i_c(\mathcal X_{\overline F_q},\Q_\ell)$ denotes compactly supported \'etale cohomology. Here (and for the rest of this article) it is implicit that the $\F_q$-points of such a DM stack $\mathcal X$ are to be counted with weights reciprocal to the size of their automorphism groups, as is understood to be the most natural way of point counting for stacks.

Achter et al. treat the case of $\mathcal X = \M_{g,n}$, the moduli of genus $g$ curves with $n$ marked points. The \'etale cohomology groups of these spaces are known in low (or high, when considering compactly supported cohomology via Poincar\'e duality) degree, and in fact become independent of $g$ for $g$ large enough; in other degrees, there also exists a collection of stable classes, though there are also other classes making up the remainder of the cohomology groups. The idea then is to make a heuristic that only the groups in ``stable'' degrees contribute meaningfully to the trace formula, with enough cancellation occuring in the non-stable classes in remaining degrees to overcome the possibility (which is known to occur for the case of $\M_g$) that there may be a very large number of non-stable classes. This heuristic may be stated precisely as follows.

\begin{heuristic}[\cite{achter-wood}, Heuristic 2]\label{mg hypothesis}
    Let $B^k(\M_g)$ denote the space of non-stable classes in $H^k_c(\M_{g,\overline\F_q},\Q_\ell)$. Then we have
    $$\lim_{g\to\infty} q^{-3g+3}\sum_{k=0}^{3g-3 - \frac{2g-3}3}(-1)^k\Tr(\Frob_q,B^k(\M_g)) = 0$$
\end{heuristic}

For a more precise description of $B^k(\M_g)$ and the stable classes of $\M_g$, and the motivation behind this heuristic, see \cite{achter-wood}; we will only need Heuristic \ref{mg hypothesis} insofar as we will make use of the computations made in \textit{loc cit.} under this assumption.

\subsection{Results} Our contribution is to carry out the analysis described in \S\ref{achter} for the case of the \textit{universal Jacobian stack}, which is a stack $\Jac_g$ which parametrizes families of curves along with a point on the Jacobian of each curve in the family. There is a natural morphism $\Jac_g \to M_g$ which remembers only the underlying family of curves, and whose fibers over $\F_q$-points of $\M_g$ are the Jacobians of the corresponding curves. The idea then is that the number of $\F_q$-points on $\Jac_g$ is the sum
$$\sum_{C \in \M_g(\F_q)}h(C),$$
where $h(C)$ is the class number of $C$, i.e. the number of $\F_q$-rational points on the Jacobian of $C$. Similarly to $\M_g$, the cohomology of the stack $\Jac_g$ contains a family of stable classes, and its low degree singular cohomology over $\C$ was computed in \cite{ebert-randall-williams}. In analogy with \cite{achter-wood}, we make the following heuristic to deal with the remaining cohomology:
\begin{heuristic}\label{average-class-number}
    Let $B_g^k$ be the space of non-stable classes in $H^k_c(\Jac_g,\Q_\ell)$. Then we have
    $$\lim_{g\to\infty} q^{-4g+3}\sum_{k=0}^{4g-3 - \frac{2g-3}3}(-1)^k\Tr(\Frob_q,B^k_g) = 0$$
\end{heuristic}
Assuming this heuristic, we prove that $\#\Jac_g(\F_q)$ grows asymptotically in $g$ as $C(q)q^{4g-3}$ for an explicit constant $C(q)$. If we also assume Heuristic \ref{mg hypothesis} and use the resulting asymptotic for $\#\M_g(\F_q)$, then the ratio $\#\Jac_g(\F_q)/\#\M_g(\F_q)$ computes the mean class number among $\F_q$-curves of genus $g$, and we find:

\begin{theoremalpha}\label{intro class number theorem}
    Let $h_g$ denote the mean class number of $\F_q$-curves of genus $g$. Under the above two heuristics, we have
    $$h_g \sim c(q)q^g$$
    as $g\to\infty$, for an explicit constant $c(q)$.
\end{theoremalpha}

The class number $h(C)$ may also be interpreted as $q$ times the residue of $Z(C,t)$ at $t=1$, which in principle gives a modicum of information about the zeros of $Z(C,t)$. Using work of Tsfasman and Vladut, we are able to show that the growth rate predicted in Theorem \ref{intro class number theorem} is sufficiently \textit{slow} so as to actually be coercive for the distribution of zeta zeros. This allows us to prove, again conditionally on our two heuristics, our main result.
\begin{theoremalpha}\label{intro zeta zero theorem}
    Under the above two heuristics, the zeta zeros of curves in $\M_g(\F_q)$ become equidistributed on the circle as $g\to\infty$.
\end{theoremalpha}
See Theorem \ref{zeta zero theorem} for a more precise statement. In Section 2, we recall the definitions of the universal Jacobian stack in the algebraic and analytic settings and prove that they are related by analytification, which is likely known to experts but to our knowledge has not been explicitly recorded in the literature. We also describe the stable classes on $\Jac_g$ which we will need for our computation. In Section 3 we demonstrate a comparison isomorphism between the singular cohomology of $\Jac_g$ over $\C$ and the \'etale cohomology of $\Jac_g$ over $\overline\F_q$, allowing us to use Ebert and Randall-Williams' computation of the low degree singular cohomology for our application. In Section 4 we carry out our trace formula computation for $\#\Jac_g(\F_q)$ and prove Theorem \ref{intro class number theorem}. In Section 5 we discuss the application to zeta zeros and prove Theorem \ref{intro zeta zero theorem}.

\subsection{Remark on heuristics} That $\M_g$ and $\Jac_g$ exhibit stability in their low degree cohomology is an instance of a widespread phenomenon in the cohomology of ``naturally occuring'' families of spaces. There is a general philosophy in the theory of moduli spaces that the geometric properties of the objects parametrized by a family of moduli spaces should only depend on the ``stable cohomology'' of this family. Heuristics \ref{mg hypothesis} and \ref{average-class-number} are two instances of this philosophy.

However, this principle is not infallible; for instance, Lipnwoski and Tsimerman prove in \cite{lipnowski-tsimerman} that the number of $\F_p$-points of $\mathcal A_g$, the moduli of principally polarized abelian varieties, is much greater than the stable cohomology of the $\mathcal A_g$ predicts. As such, if one believes Heuristics \ref{mg hypothesis} and \ref{average-class-number}, then a possible interpretation of the situation is that there is something special about the image of the Torelli map $\M_{g,\F_p} \hookrightarrow \mathcal A_{g,\F_p}$ when compared with $\mathcal A_{g,\F_p}$ as a whole.

\subsection{Acknowledgements} This project came about due to a number of suggestions by Peter Sarnak, and the author is very grateful to him for his advisorship and his insights along the way. The author would also like to thank Will Sawin for many helpful comments, suggestions, and corrections.

\section{Universal Jacobians and stable cohomology}

The universal Jacobian stack can be formulated both over the ``analytic'' category $\mathbf{An}$ of complex analytic spaces, and over the ``algebraic'' category $\mathbf{Sch}$ of schemes. We need and will define both of these notions, and part of our job will be to prove that the former holomorphic stack is the analytification of the latter algebraic stack.

\subsection{Definitions}
Analytically, we define a stack $\Jacan_g$ whose fiber over an analytic space $X$ is the groupoid of families $E\to X$ of genus $g$ Riemann surfaces over $X$, along with a holomorphic section $s:X \to \mathrm{Pic}(E/X)$ of the degree 0 Picard bundle of the family; isomorphisms are isomorphisms of families over $X$ respecting sections.

Algebraically, the definition is slightly more involved. First we define an auxiliary stack $\J_g$ whose fiber over a scheme $S$ is the groupoid of smooth relative curves $C \to S$ of genus $g$, equipped with a line bundle $L$ on $C$; isomorphisms $(C,L) \to (C',L')$ are $S$-isomorphisms $f:C\to C'$ along with an isomorphism $f^*L' \to L$.

The stack $\J_g$ is too large to be our universal Jacobian, as each object of $\J_g$ contains a copy of $\mathbb G_m$ in its automorphism group owing to the scalar multiplication action of $\mathbb G_m$ on the line bundle $L$. In particular, $\J_g$ is not Deligne-Mumford. To remedy this, we define $\Jac_g$ to be the rigidification $\J_g\!\! \fatslash \mathbb G_m$ (here we reverse the notation of \cite{melo-viviani}, as for us the rigidified stack will be the fundamental object).

Both $\Jacan_g$ and $\Jac_g$ are Deligne-Mumford stacks over their respective categories. $\Jac_g$ has a smooth, proper, and representable morphism to $\M_g$ which forgets the line bundle, and the fiber of $\Jac_g$ over a $k$-point for a field $k$ is the Jacobian of the curve of the corresponding $k$-point of $\M_g$. For a more thorough review of the analytic and algebraic universal Jacobian stacks, see \cite{ebert-randall-williams} and \cite{melo-viviani} respectively.

\subsection{Analytification}
Ultimately, we wish to make a computation regarding Jacobians of curves over $\F_q$, meaning that we will need to understand the \'etale cohomology of $\Jac_g$ over $\F_q$, rather than the singular cohomology of the analytic stack $\Jacan_g$. Showing that Ebert and Randall-Williams' stable ring also gives the low degree \'etale cohomology of $\Jac_{g,\F_q}$ will be the purpose of the next section; before doing so, it is necessary that we relate $\Jacan_g$ and $\Jac_g$ by showing that the former is the analytification of the complex points of the latter.

\begin{lemma}\label{artinian}
    Let $f:\mathcal X \to \mathcal Y$ be a 1-morphism of Deligne-Mumford stacks over $\textbf{An}$. Suppose that for each local Artinian scheme $S$ which is finite-type over $\C$, the induced functor $f:\mathcal X(S^\an) \to \mathcal Y(S^\an)$ is an equivalence of categories. Then $f$ is an isomorphism.
\end{lemma}

\begin{proof}
    See \cite[Lemma 7.1]{hall}.
\end{proof}

\begin{theorem}\label{analytification}
    For $g\geq2$, the natural morphism $\Jac_{g,\C}^\an \to \Jacan_g$ is an isomorphism.
\end{theorem}

That this map is an isomorphism has been suggested by various authors (see \cite[\S4.5]{ebert-randall-williams}, \cite[\S1.1]{melo-viviani}), and is likely known to experts.

\begin{proof}
    We will show that $\Jac_{g,\C}^\an \to \Jacan_g$ satisfies the assumptions of Lemma \ref{artinian}. Let $S$ be a local Artinian finite-type $\C$-scheme. Note that $S$ is proper over $\C$; leveraging this, the idea is to show we have an equivalence of $S^\an$-points by using various GAGA theorems.

    By \cite[Theorem C]{hall}, analytification induces an equivalence of categories
    $$\Jac_{g,\C}(S) \to \Jac_{g,\C}^\an(S^\an).$$
    The rigidification $\Jac_{g,\C}$ has $S$-points in bijection with pairs of smooth genus $g$ $S$-curves along with an isomorphism class of line bundle, with the obvious isomorphisms between such pairs. Thus, the functor
    $$\Jac_{g,\C}(S) \to \Jac_{g,\C}^\an(S^\an) \to \Jacan_g(S^\an)$$
    can be described by sending $(C/S,L)$ to the analytification $(C^\an/S^\an,L^\an)$. By \cite[Theorem A]{hall}, any genus $g$ $S^\an$-curve is the analytification of an $S$-curve (note that any such $S^\an$-curve has finite automorphisms since $g\geq2$), and any morphism of $S^\an$-curves is algebraic by Serre's GAGA (any $S^\an$-curve is proper over $\C$). Moreover, any line bundle on an $S^\an$-curve is isomorphic to an algebraic one. In total, the above functor is fully faithful and essentially surjective, and hence is an equivalence.
\end{proof}

\subsection{Stable cohomology}
Here we review Ebert and Randall-Williams' computation of the stable cohomology of $\Jacan_g$. To state their theorem, we define the graded algebra
$$R = \Q[\nu_{ij} \mid i\geq-1, \,\, j\geq0, \,\, (i,j) \neq (-1,0),(0,1)],$$
where we let the generator $\nu_{ij}$ have degree $2i+2j$. We will call this algebra the \textit{tautological ring}, and the $\nu_{ij}$ are called \textit{tautological classes}.

\begin{theorem}[\cite{ebert-randall-williams}, Theorem F]\label{stability}
    Let $g \geq 6$. For all $i\geq -1$, $j\geq0$, there exist classes $\nu_{ij} \in H^{2i+2j}(\Jacan_g,\Q)$ so that the induced algebra map
    $$R \to H^\bullet(\Jacan_g,\Q)$$
    is an isomorphism in degrees $\bullet \leq \frac{2g-3}3$.
\end{theorem}

Let us give a brief description of the construction of these tautological classes $\nu_{ij}$, which will be necessary for our application. In view of Theorem \ref{analytification}, we will give a slightly different presentation than in \cite{ebert-randall-williams}, opting to construct the $\nu_{ij}$ algebraically and factor through the analytification to obtain classes on $\Jacan_g$, rather than the direct topological approach taken in \textit{loc cit}.

We start by considering the universal curve $\mathcal C \to \J_g$, which is pulled back from the universal curve over $\M_g$. Here there are two relevant sheaves on $\mathcal C$: the relative dualizing sheaf $\omega_{\mathcal C/\J_g}$, and the universal line bundle $\mathcal L$. Letting $\pi:\J_g^\an \to \Jacan_g$ be the rigidification map, we can then define our tautological classes as
$$\nu_{ij} = \pi_!(c_1(\omega_{\mathcal C/\J_g})^{i+1}c_1(\mathcal L)^j) \in H^{2i+2j}(\Jacan_g,\Q),$$
where in the above expression $\omega_{\mathcal C/\J_g}$ and $\mathcal L$ are to denote the analytified sheaves on $\J_g^\an$. The key feature of this construction for our purposes is that the classes $\nu_{ij}$ factor through the Chow ring of $\Jac_g$, which will be important when we later describe the Frobenius action on the $\nu_{ij}$ in characteristic $p$.

\section{A comparison isomorphism for cohomology}

We now want to demonstrate an isomorphism between the singular cohomology of $\Jacan_g$, which we understand in low degrees, and the \'etale cohomology of $\Jac_g$ over $\overline\F_q$. In light of Theorem \ref{analytification}, there is an essentially standard comparison isomorphism
$$H^i(\Jacan_g,\Q_\ell) \cong H^i(\Jac_{g,\C},\Q_\ell),$$
so the crux of the matter is comparing the cohomology groups of $\Jac_{g,\C}$ and $\Jac_{g,\overline\F_q}$.

In \cite{achter-wood}, such a comparison between characteristic 0 and characteristic $p$ cohomology is proven for $\M_g$. The key input is that the Deligne-Mumford compactification $\overline\M_g$ admits a cover by a scheme which is smooth and proper over $\Z_p$, and whose fiber above $\M_g$ is the complement of a normal crossings divisor. We will leverage this cover of $\overline\M_g$ to prove our desired result for $\Jac_g$.

\begin{lemma}[\cite{achter-wood}, Lemma 9]\label{mg-cover}
There exists a smooth, proper scheme $\overline Y_g$ over $\Z_p$, a normal crossings divisor $D$ on $\overline Y_g$, and a finite group $G$ acting on $\overline Y_g$ so that $G$ preserves $D$, and so that $[Y_g/G] = \M_{g,\Z_p}$, where $Y_g = \overline Y_g \setminus D$.
\end{lemma}

To apply Lemma \ref{mg-cover} to the situation at hand, we prove a base change theorem for fibrations over a base with a suitable compactification.

\begin{proposition}\label{specialization prop}
    Let $A$ be a Henselian DVR of characteristic zero, and let $\bar\eta$ and $\bar s$ be geometric generic and special points of $A$, respectively. Let $X$ be a smooth, proper $A$-scheme, and let $D$ be a divisor on $X$ which is normal crossings over $A$.

    Let $U = X \setminus D$, and suppose that we have a smooth proper morphism $\pi:V \to U$. Moreover, suppose that a finite group $G$ acts compatibly on $V$ and $U$ by $A$-morphisms. Then for all primes $\ell$ distinct from the residue characteristic of $A$, there is a $G$-equivariant isomorphism
    \begin{gather}\label{specialization}
        H^i(V_{\bar\eta},\Q_\ell) \to H^i(V_{\bar s},\Q_\ell)
    \end{gather}
    given by the specialization map for cohomology.
\end{proposition}

We remark this is essentially \cite[Prop. 7.7]{ellenberg-venkatesh-westerland} for the case of a smooth and proper morphism $\pi$, rather than a finite \'etale cover. Our proof strategy is the same, using Deligne's theory of vanishing cycles, with minor modifications.

\begin{proof}
    Observe that it suffices to demonstrate the isomorphism (\ref{specialization}) with $\Z/\ell^n\Z$ coefficients for all $n$. Let $\mathcal F$ be the constant sheaf $\Z/\ell^n\Z$ on $V$. Then we have a Leray spectral sequence
    $$H^p(U_{\bar\eta},R^q\pi_*\mathcal F) \Longrightarrow H^{p+q}(V_{\bar\eta},\mathcal F),$$
    and similarly for the case of $\bar s$-fibers. In light of these spectral sequences, it suffices to demonstrate $G$-equivariant isomorphisms
    $$H^p(U_{\bar\eta},R^q\pi_*\mathcal F) \to H^p(U_{\bar s},R^q\pi_*\mathcal F)$$
    which are functorial in $\mathcal F$.

    Since $\pi$ is smooth and proper, the sheaf $R^q\pi_*\mathcal F$ is locally constant. Let $j:U \to X$ be the inclusion, and let $\mathcal G^q = j_!R^q\pi_*\mathcal F$. $U_{\bar\eta}$ is a smooth variety, so by Poincar\`e duality and the definition of compactly supported \'etale cohomology we have
    $$H^p(U_{\bar\eta},R^q\pi_*\mathcal F) \cong H^{2d-p}_c(U_{\bar\eta},R^q\pi_*\mathcal F)^\vee = H^{2d-p}(X_{\bar\eta},\mathcal G^q)^\vee,$$
    where $d = \dim X_{\bar\eta}$ and the duals are as $\Z/\ell^n\Z$-modules. The above also holds for $\bar s$ in place of $\bar\eta$.

    The specialization map from the cohomology of $X_{\bar\eta}$ to the cohomology of $X_{\bar s}$ is part of an exact triangle in terms of the (hyper)cohomology of Deligne's complex of vanishing cycles (see \cite[Exp. XIII, 2.1.8.9]{sga7}):
    $$\dots \to H^i(X_{\bar\eta},\mathcal G^q) \to H^i(X_{\bar s},\mathcal G^q) \to H^i(X_{\bar s},R\Phi_{\bar\eta}(\mathcal G)) \to \dots$$
    Since $R^q\pi_*\mathcal F$ is locally constant and is vacuously tamely ramified along $D$ by virtue of $A$ being characteristic 0, by \cite[Exp. XIII, Lemme 2.1.11]{sga7} the aptronymous complex of vanishing cycles, in fact, vanishes. As such, the specialization map in this case is an isomorphism, and so we find
    \begin{gather}\label{composite isomorphism}
        H^p(U_{\bar\eta},R^q\pi_*\mathcal F) \cong H^{2d-p}(X_{\bar\eta},\mathcal G^q)^\vee \cong H^{2d-p}(X_{\bar s},\mathcal G^q)^\vee \cong H^p(U_{\bar s},R^q\pi_*\mathcal F)
    \end{gather}
    Since Poincar\'e duality and the specialization map are both functorial, these isomorphisms are functorial in $\mathcal F$; moreover, the functoriality of these maps also ensures $G$-equivariance, and so we have a $G$-isomorphism. Finally, note that the specialization map is compatible with the pairings for Poincar\'e duality at $\bar\eta$ and $\bar s$, so the composite isomorphism (\ref{composite isomorphism}) is in fact the specialization map $H^p(U_{\bar\eta},R^q\pi_*\mathcal F) \to H^p(U_{\bar s},R^q\pi_*\mathcal F)$. Then since specialization is itself functorial, it follows that the isomorphism (\ref{specialization}) obtained from the Leray spectral sequence is the specialization map for $V$.
\end{proof}

Write $Z_g = Y_g \times_{\M_{g}} \Jac_{g,\Z_p}$, which is a scheme since $\Jac_g \to \M_g$ is representable. Since $Y_g \to \M_{g,\Z_p}$ is a quotient map by a finite group $G$, the action of $G$ can be pulled back to $Z_g$ so that $Z_g \to \Jac_{g,\Z_p}$ is a finite Galois cover with Galois group $G$. Moreover, $Z_g \to Y_g$ is smooth and proper since $\Jac_g \to \M_g$ is smooth and proper.

We equip the tautological ring $R$ of Theorem \ref{stability} with the following Frobenius action: we specify that $\Frob_q$ acts on the class $\nu_{ij}$ by multiplication by $q^{i+j}$.

\begin{corollary}\label{char p stability cor}
    For each $k \leq \frac{2g-3}3$, there is an isomorphism of Frobenius modules
    \begin{gather}\label{char p stability}
        R^k_{\Q_\ell} \to H^k(\Jac_{g,\overline\F_q},\Q_\ell)
    \end{gather}
\end{corollary}

\begin{proof}
    We define the images of the $\nu_{ij}$ in $H^k(\Jac_{g,\overline\F_q},\Q_\ell)$ as we did in $H^k(\Jacan_g,\Q)$ in terms of the (algebraic) sheaves $\omega_{\mathcal C/\J}$ and $\mathcal L$. Note that then the morphism thusly defined then factors through the Chow ring of $\Jac_{g,\overline\F_q}$, meaning that the images of the $\nu_{ij}$ are of Tate type and have Frobenius eigenvalue $q^{i+j}$. Thus, we have a morphism as in \ref{char p stability} which is Frobenius equivariant, and we need only to check that it is an isomorphism of $\Q_\ell$-modules.

    Choose an embedding $\Q_p \to \C$, yielding a $\C$-point of $\Spec\Z_p$. By Proposition \ref{specialization prop} applied to the morphism $Z_g \to Y_g$, there is a $G$-equivariant isomorphism
    $$H^k(Z_{g,\C},\Q_\ell) \to H^k(Z_{g,\overline\F_q},\Q_\ell)$$
    Upon taking $G$-invariants of the above, Hochschild-Serre spectral sequence then guarantees an isomorphism
    $$H^k(\Jac_{g,\C},\Q_\ell) \to H^k(\Jac_{g,\overline\F_q},\Q_\ell)$$
    There is also a standard comparison isomorphism
    $$H^k(Z_{g,\C},\Q_\ell) \to H^k(Z_{g,\C}^\an,\Q_\ell)$$
    between \'etale and singular cohomology which is functorial and therefore $G$-equivariant, and as such descends to an isomorphism
    $$H^k(\Jac_{g,\C},\Q_\ell) \to H^k(\Jacan_{g,\C},\Q_\ell)$$
    owing again to the Hochschild-Serre spectral sequence. The composite isomorphism
    $$H^k(\Jac_{g,\overline\F_q},\Q_\ell) \to H^k(\Jac_{g,\C},\Q_\ell) \to H^k(\Jacan_{g,\C},\Q_\ell)$$
    is functorial, and so in particular the classes $\nu_{ij}$ in $H^k(\Jac_{g,\overline\F_q},\Q_\ell)$ map to the corresponding tautological classes $H^k(\Jacan_{g,\C},\Q_\ell)$ since the bundles $\omega_{\mathcal C/\J}$ and $\mathcal L$ can be defined integrally. Thus, the above composition identifies (\ref{char p stability}) with the base change of the isomorphism in Theorem \ref{stability} to $\Q_\ell$.
\end{proof}

Let us also make note of the fact that even for $k > \frac{2g-3}3$, there is a (Frobenius-equivariant) map $R^k_{\Q_\ell} \to H^k(\Jac_{g,\overline\F_q},\Q_\ell)$, whose image and cokernel we will consider in the sequel.

\section{Computing with stable classes}

We are now prepared to make our main computation. In order to invoke the Behrend-Grothendieck-Lefschetz trace formula, we must reformulate Theorem \ref{char p stability cor} in terms of compactly supported cohomology. Since $\Jac_{g,\overline\F_q}$ is smooth, this is a simple matter of applying Poincar\'e duality, and we obtain a map
$$R^{2d_g-k}_{\Q_\ell} \to H^k_c(\Jac_{g,\overline\F_q},\Q_\ell)$$
which is an isomorphism for $2d \geq k \geq 2d - \frac{2g-3}3$, where $d_g = \dim \Jac_{g,\overline \F_q} = 4g-3$. Note that this isomorphism is \textit{not} Frobenius equivariant for the given Frobenius action on $R$ due to the Tate twist implicit in Poincar\'e duality; instead, the image of $\nu_{ij} \in R^{2d_g-k}_{\Q_\ell}$ will have Frobenius eigenvalue $q^{k/2}$.

Let us denote by $A_g^k$ the image of $R^{2d_g-k}_{\Q_\ell} \to H^k_c(\Jac_{g,\overline\F_q},\Q_\ell)$, and let $B_g^k$ be the cokernel, so that
$$\Tr(\Frob_q,H^k_c(\Jac_{g,\overline\F_q},\Q_\ell)) = q^{k/2}\cdot\dim A_g^k + \Tr(\Frob_q,B_g^k)$$
Then, noting that $R$ and $A_g$ are concentrated in even degrees, the trace formula for $\Jac_{g}$ reads
\begin{align*}
    \#\Jac_g&(\F_q) = \sum_{k=2d_g-\frac{2g-3}3}^{2d_g}q^{k/2}\cdot\dim A_g^k + \sum_{k=0}^{2d_g-\frac{2g-3}3}(-1)^k(q^{k/2}\cdot\dim A_g^k + \Tr(\Frob_q,B_g^k)) \\
    &= q^d_g\sum_{k=0}^{\frac{2g-3}3}q^{-k/2}\cdot\dim R_{\Q_\ell}^{k} + q^{d_g}\sum_{k=\frac{2g-3}3}^{2d_g}q^{-k/2}\cdot\dim A_{g}^k + \sum_{k=0}^{2d_g-\frac{2g-3}3}(-1)^k\Tr(\Frob_q,B_g^k)
\end{align*}
The first sum converges to $HS_R(q^{-1/2})$, where $HS_R$ is the Hilbert series of the graded ring $R$. By Heuristic \ref{average-class-number}, the third sum is $o(q^{d_g})$, and the second sum can be bounded as
$$\sum_{k=\frac{2g-3}3}^{2d_g}q^{-k/2}\cdot\dim A_{g}^k \leq \sum_{k=\frac{2g-3}3}^{\infty}q^{-k/2}\cdot\dim R_{\Q_\ell}^k$$
which tends to $0$ as $g\to\infty$ since, say, the Hilbert series for $R$ has a positive radius of convergence. Thus, in aggregate we find that
$$\#\Jac_g(\F_q) \sim q^{4g-3}HS_R(q^{-1/2})$$
as $g\to\infty$. With this asymptotic point count in hand, we may conclude as follows.

\begin{theorem}\label{main class number theorem}
    Assume Heuristic \ref{mg hypothesis} and Heuristic \ref{average-class-number}. Let $h_g$ denote the mean class number of curves in smooth $\F_q$-curves of genus $g$. Then we have
    $$h_g \sim c(q)q^g$$
    as $g\to\infty$, where
    $$c(q) = \frac{1}{1-q^{-1}}\cdot\prod_{k=2}^\infty\Big(\frac{1}{1-q^{-k}}\Big)^{k+1}$$
\end{theorem}

\begin{proof}
    By the analysis carried out in \cite[\S4]{achter-wood}, we have
    $$\#\M_g(\F_q) \sim q^{3g-3} \prod_{k=1}^\infty(1-q^{-k})^{-1}$$
    as $g\to\infty$. Since $R$ is a free algebra in the variables $\nu_{ij}$, we find that
    $$HS_R(t) = \frac{1}{(1-t^2)^{2}}\prod_{k=2}^\infty\frac1{(1-t^{2k})^{k+2}}$$
    Combining the above, we have
    $$h_g = \frac{1}{\#\M_g(\F_q)}\sum_{C\in\M_g(\F_q)}h(C) = \frac{\#\Jac_g(\F_q)}{\#\M_g(\F_q)} \sim q^g \cdot \frac{1}{1-q^{-1}}\cdot\prod_{k=2}^\infty\Big(\frac{1}{1-q^{-k}}\Big)^{k+1}$$
    as desired.
\end{proof}

\section{The distribution of zeta zeros}

In this section, we discuss present some applications of Theorem \ref{main class number theorem} to the asymptotic distributions of the zeros of zeta functions of curves. For a fixed curve $C \in \M_g(\F_q)$, let $\omega_1,\dots,\omega_{2g}$ be the roots of the zeta function $Z(C,t)$. Define a measure on the unit circle in $\C$ by
$$\mu_C = \frac\pi{g}\sum_{i=1}^{2g}\delta_{\omega_i/\sqrt q}$$
We can then define a measure encoding the distribution of zeta zeros for all curves in $\M_g(\F_q)$ by
$$\mu_g = \frac1{\#\M_g(\F_q)}\sum_{C \in \M_g(\F_q)}\mu_C$$
It would then be interesting to determine what can be said about the weak limit of the $\mu_g$ as $g\to\infty$ (if it exists).

\subsection{Asymptotically exact families} To relate our average class number heuristic to studying the limit of the $\mu_g$, we will use the framework of Tsfasman and Vladut, which is best treated in \cite{tsfasman-vladut}. We will quickly review the key notions here. Let $C_1,C_2,\dots$ be an infinite family of curves over $\F_q$, no two of which are isomorphic. We may consider the limits
$$\phi_m = \lim_{i\to\infty}\frac{B_m(C_i)}{g(C_i)}$$
where $B_m(C)$ denotes the number of degree $m$ prime divisors of $C$; if all the $\phi_m$ exist, we say that the family $\{C_i\}$ is \textit{asymptotically exact}. Tsfasman-Vladut relate the asymptotic behavior of $h(C_i)$ to the $\phi_m$ in the following way:

\begin{theorem}[{\cite[Theorem D']{tsfasman-vladut}}] \label{brauer-siegel}
    Let $\{C_i\}$ be an asymptotically exact family of curves over $\F_q$. Then we have
    $$\lim_{i\to\infty}\frac{\log_q h(C_i)}{g(C_i)} = 1 + \sum_{m=1}^\infty\phi_m\log_q\frac{1}{1-q^{-m}}$$
\end{theorem}

In other words, the growth rate of the class numbers of the $C_i$ should be exponential in $g$, with base equal to $q$ to the power of the right hand side. Comparing with Theorem \ref{main class number theorem}, we are led to the rough prediction that for an ``average'' family of curves, the exponential growth rate of $h(C_i)$ is as small as possible, and the $\phi_m$ are all zero.

This prediction is rather striking, for the following reason. In studying the arithmetic statistics of curves over finite fields in the $g\to\infty$ limit, one expects in some sense that the central difficulty or complexity comes from the dominance of high gonality curves as the genus grows. Indeed, many questions regarding, say, point counts become far more tractable when restricting to low gonality strata of $\M_g$; see e.g. \cite{hyperelliptic}, \cite{trigonal}.

However, it can easily be seen that a family of curves of bounded gonality (which some would call a \textit{tame} family) has $\phi_m = 0$ for all $m$. Combining this observation with the above prediction for the random curve, we find that, at least with respect to their class numbers, random curves in $\M_g(\F_q)$ behave as the low-gonality curves do, despite the low-gonality strata being positive codimension.

The following theorem of Tsfasman and Vladut can be used to relate the above tame behavior to the limiting distribution of zeta zeros:

\begin{theorem}[{\cite[Corollary B']{tsfasman-vladut}}] \label{explicit-formula}
    Let $\mathcal C = \{C_i\}$ be an asymptotically exact family of curves over $\F_q$. Then the sequence $\mu_{C_1},\mu_{C_2},\dots$ converges weakly to a measure $\mu_{\mathcal C}$ which is absolutely continuous with respect to the Lebesgue measure on the unit circle and is given by the density
    $$M_{\mathcal C}(t) = 1 - \sum_{m=1}^\infty m\phi_m h_m(t)$$
    for some explicit functions $h_m(t)$ depending only on $m$. In particular, when all the $\phi_m$ are zero, the zeta zeros of the $C_i$ equidistribute on the circle with respect to the Lebesgue measure.
\end{theorem}

Thus, the prediction that class numbers of random curves behave tamely on average as $g\to\infty$ yields a very strong consequence for how the zeta zeros of a random curve distribute, at least when such curves are viewed in families. In fact, we can also adapt the methods of Tsfasman in \cite{tsfasman} to make a more uniform statement about all of $\M_g$ and the measures $\mu_g$ as $g\to\infty$.

\begin{theorem}\label{zeta zero theorem}
    Assume Heuristic \ref{mg hypothesis} and Heuristic \ref{average-class-number}. Then the $\mu_{g}$ converge weakly to the uniform measure on the unit circle.
\end{theorem}

\begin{proof}
    Write for the average class number of curves in $\M_g(\F_q)$
    $$h_g = \frac{1}{\#\M_g(\F_q)}\sum_{C \in \M_g(\F_q)}h(C)$$
    Fix some $\varepsilon > 0$, and decompose $\M_g(\F_q)$ into two sets $S_{g,\varepsilon}$ and $T_{g,\varepsilon}$ such that:
    \begin{itemize}
        \item $|\frac{\#T_{g,\varepsilon}}{\#\M_g(\F_q)} - \varepsilon| < \frac{1}{\#\M_g(\F_q)}$
        \item No curve in $S_{g,\varepsilon}$ has greater class number than any curve in $T_{g,\varepsilon}$.
    \end{itemize}
    In other words, $S_{g,\varepsilon}$ contains the bottom $1-\varepsilon$ proportion of the curves in $\M_g(\F_q)$ when sorted by class number, and $T_{g,\varepsilon}$ contains the remainder. Then we can decompose $\mu_g$ as $\mu_g = \mu_g^{S,\varepsilon} + \mu_g^{T,\varepsilon}$, where
    $$\mu_g^{S,\varepsilon} = \frac{1}{\#\M_g(\F_q)}\sum_{C \in S_{g,\varepsilon}}\mu_C$$
    and $\mu_g^{T,\varepsilon}$ is defined analogously.

    Suppose that for some $\delta>0$ there were an infinite increasing sequence $\{g_i\}$ so that for all $i$,
    $$\sup_{C \in S_{g_i,\varepsilon}} h(C) \geq q^{g_i(1 + \delta)}$$
    Then for all such $i$ we would have
    \begin{align*}
        h_{g_i} &\geq \frac{1}{\#\M_{g_i}(\F_q)}\sum_{C \in T_{g_i,\varepsilon}}h(C) \geq \varepsilon \cdot q^{g_i(1+\delta)}
    \end{align*}
    Thus, if the conclusion of Theorem \ref{main class number theorem} holds, then it must be the case that there is no such $\delta$, and so in particular we have
    \begin{align}\label{S-bound}
        \limsup_{g\to\infty}\sup_{C \in S_{g,\varepsilon}} \log h(C) \leq g\log q
    \end{align}
    
    Define a family of curves $\mathcal C = \{C_i\}$ by first inserting the curves of $S_{1,\varepsilon}$ into the sequence in any order, followed by the curves of $S_{2,\varepsilon}$ in any order, and so on. We will show that, under our heuristics, $\mathcal C$ is asymptotically exact with each $\phi_m(\mathcal C) = 0$.

    Let $C$ be an $\F_q$-curve. Using that $h(C) = L(C,1)$, it is an elementary computation that for each $N$ we have
    $$\frac{\log h(C)}{g(C)} - \log q = \frac1{g(C)}\sum_{j=1}^N \frac{q^{-j}}jN_j(C) - \frac1{g(C)}\sum_{j=1}^N\frac{1+q^{-j}}j - \frac1{g(C)}\sum_{j=N+1}^\infty\frac{q^{-j/2}}j\Omega(C,j),$$
    where $N_j(C)$ is the number of $\F_{q^j}$-points of $C$, and $\Omega(C,j)$ is the sum of the $j$th powers of the zeta zeros of $C$ (this is Lemma 2 of \cite{tsfasman}). Let $N = N(g)$ be a function of $g$ increasing to infinity slowly enough that $N(g)/\log g \to 0$. Then the rightmost term of the above goes to 0 as $g \to \infty$ (so then $N \to \infty)$; moreover, we have
    $$\frac1{g(C)}\sum_{j=1}^{N}\frac{1+q^{-j}}j \leq \frac{2}{g}(1 + \log N) \to 0$$
    Thus, we have
    \begin{align*}
        \lim_{i\to\infty}\Big(\frac{\log h(C_i)}{g(C_i)} - \log q\Big) &= \lim_{i\to\infty}\frac1{g(C_i)}\sum_{j=1}^{N(g(C_i))} \frac{q^{-j}}jN_j(C_i) \\
        &= \lim_{i\to\infty} \frac1{g(C_i)}\sum_{j=1}^{N}\frac{q^{-j}}j\sum_{m \mid j}m\cdot B_m(C_i) \\
        &= \lim_{i\to\infty}\sum_{m=1}^{N}B_m(C_i)\sum_{j=1}^{N/m}\frac{q^{-jm}}{j} \\
        &= \lim_{i\to\infty}\Big(\sum_{m=1}^N \frac{B_m(C_i)}{g(C_i)} \log\frac{1}{1-q^{-m}} - \frac1{g(C_i)}\sum_{m=1}^N B_m\sum_{i=(N/m)+1}^\infty\frac{q^{-im}}i\Big)
    \end{align*}
    Let us deal with the second term of the last expression. It can be bounded as
    \begin{align*}
        \frac1{g(C_i)}\sum_{m=1}^N B_m\sum_{i=(N/m)+1}^\infty\frac{q^{-im}}i &\leq \frac1{g(C_i)}\sum_{m=1}^N\frac{N_m(C_i)}m\cdot\frac{q^{-N-m+1}}{(\frac Nm)(1-q^{-m})} \\
        &\leq \frac{1}{g(C_i)}\sum_{m=1}^N(q^m+1+2g(C_i)q^{m/2})\frac{q^{-N-m+1}}{N(1-q^{-m})}
    \end{align*}
    where the last inequality uses the Weil conjectures. This last expression goes to 0 as $i\to\infty$ (since then $g(C_i)$ and $N$ go to infinity), so we have
    $$\lim_{i\to\infty}\Big(\frac{\log h(C_i)}{g(C_i)} - \log q\Big) = \lim_{i\to\infty}\sum_{m=1}^N \frac{B_m(C_i)}{g(C_i)} \log\frac{1}{1-q^{-m}}.$$
    On the one hand, under Theorem \ref{main class number theorem}, by (\ref{S-bound}) the left hand side is bounded above by 0. On the other hand, the right hand side is non-negative. Thus, both limits exist and are 0. In particular, the limits
    $$\lim_{i\to\infty}\frac{B_m(C_i)}{g(C_i)}$$
    all exist and are 0. Thus, $\mathcal C$ is asymptotically exact, and it follows from Theorem \ref{explicit-formula} that the $\mu_{C_i}$ converge to the uniform measure.

    Let $\mathcal L$ denote the uniform probability measure on the circle; then it follows from the above that the $\mu_g^{S,\varepsilon}$ converge to $(1-\varepsilon)\mathcal L$. Since $\mu_g^{S,\varepsilon}$ and $\mu_g^{T,\varepsilon}$ are both positive measures with $\mu_g^{T,\varepsilon}$ having total mass approaching $2\pi\varepsilon$ as $g\to\infty$, the decomposition $\mu_g = \mu_g^{S,\varepsilon} + \mu_g^{T,\varepsilon}$ for each $\varepsilon$ implies that $\mu_g$ converges to $\mathcal L$ as desired.
\end{proof}

To conclude, we discuss a possible interpretation of the constant $c(q)$ appearing in Theorem \ref{main class number theorem}. The constant
$$c(q) = \frac{1}{1-q^{-1}}\cdot\prod_{k=2}^\infty\Big(\frac{1}{1-q^{-k}}\Big)^{k+1}$$
bears some superficial similarity to the sum on the right hand side of Theorem \ref{brauer-siegel}, which contributes a multiplicative term of
$$\prod_{k=1}^\infty \Big(\frac1{1-q^{-k}}\Big)^{\phi_k\cdot g}$$
to the asymptotic growth rate of $h(C_i)$ for $\{C_i\}$ an asymptotically exact family. Theorem \ref{main class number theorem} predicts that curves in $\M_g(\F_q)$ behave on average like a family with all $\phi_m = 0$, and so naturally the only exponential term in $g$ is the $q^g$ factor; however, it could be interesting to interpret the exponents appearing in $c(q)$ in terms of the relative quantities of degree $k$ prime divisors on an average curve.

\printbibliography

\end{document}